\documentclass[12pt]{article}

\usepackage{amsmath,amssymb,latexsym}
\usepackage{mathrsfs}
\usepackage{amsthm} 
\usepackage{color}

 \newtheorem{thm}{Theorem}[section]

 \newtheorem{prop}[thm]{Proposition}    
 \newtheorem{ex}[thm]{Example}
\theoremstyle{definition}

 \numberwithin{equation}{section}

\newcommand{\h}{{\mathcal H}}
\newcommand{\mn}{\mathbb N}

\begin{document}

\title{On a characterization of Riesz bases \\ via biorthogonal sequences}

\author{Diana T. Stoeva}

\maketitle 

\begin{abstract}
 It is well known that a sequence in a Hilbert space is a Riesz basis if and only if it is a complete Bessel sequence with biorthogonal sequence which is also a complete Bessel sequence. Here we prove that the completeness of one (any one) of the biorthogonal sequences can be omitted in the characterization. 
\end{abstract}



\section{Introduction}
A {\it Riesz basis} for a separable Hilbert space $\h$ is a sequence of the form  $(Ve_k)_{k=1}^\infty$ with $(e_k)_{k=1}^\infty$ being an orthonormal basis for $\h$ and $V$ being a bounded bijective operator from $\h$ onto $\h$.  
 Riesz bases were introduced by Bari \cite{Bari46, Bari51} and already in \cite{Bari51} many properties and equivalent characterizations were determined.
Below we collect the standard equivalences of Riesz bases from \cite{Cbook,Hbook,Ybook}, some of which appeared already in \cite{Bari51, Gelfand51}:

\begin{thm} \label{rbequiv}
 For a sequence $(f_k)_{k=1}^\infty$ in a Hilbert space $\h$, the following conditions are equivalent:

\begin{itemize}

\item[{\rm ($\mathcal{R}_1$)}] $(f_k)_{k=1}^\infty$ forms a Riesz basis for $\h$.

\item[{\rm ($\mathcal{R}_2$)}] There is an equivalent inner product on $\h$ (i.e., an inner product generating an equivalent norm on $\h$), with respect to which  $(f_k)_{k=1}^\infty$ is an orthonormal basis for $\h$.

\item[{\rm ($\mathcal{R}_3$)}]
 $(f_k)_{k=1}^\infty$ is complete in $\h$ and there exist positive constants $A$ and $B$ so that 
$$A\sum |c_k|^2 
\leq \|\sum c_k f_k \|^2 
\leq B\sum |c_k|^2 $$
for every finite scalar sequence $(c_k)$ (and hence for every $(c_k)_{k=1}^\infty\in\ell^2$).

\item[{\rm ($\mathcal{R}_4$)}]
 $(f_k)_{k=1}^\infty$ is complete in $\h$ and its Gram matrix 
$(\langle f_k, f_j\rangle)_{j,k=1}^\infty$
determines a bounded bijective operator on $\ell^2$. 

\item[{\rm ($\mathcal{R}_5$)}]
$(f_k)_{k=1}^\infty$ is complete in $\h$ and it has a complete biorthogonal sequence $(g_k)_{k=1}^\infty$ so that 
$ \sum_{k=1}^\infty |\langle f, f_k\rangle|^2 <\infty \mbox{ and } 
\sum_{k=1}^\infty |\langle f, g_k\rangle|^2 <\infty$ for every $f\in\h$. 

\item[{\rm ($\mathcal{R}_6$)}]
 $(f_k)_{k=1}^\infty$ is a complete Bessel sequence in $\h$ and it has a biorthogonal sequence $(g_k)_{k=1}^\infty$ which is also a complete Bessel sequence in $\h$.

\item[{\rm ($\mathcal{R}_7$)}]
 $(f_k)_{k=1}^\infty$ is a bounded unconditional basis for $\h$.

\item[{\rm ($\mathcal{R}_8$)}]
$(f_k)_{k=1}^\infty$ is a basis for $\h$ such that
$\sum_{k=1}^\infty c_k f_k$ converges in $\h$ if and only if $\sum_{k=1}^\infty |c_k|^2<\infty$.

\end{itemize}
\end{thm}

The purpose of this paper is to show that one may remove the condition for completeness of  one (any one) of the sequences $(f_k)_{k=1}^\infty$ and $(g_k)_{k=1}^\infty$  in ($\mathcal{R}_6$) (and respectively in ($\mathcal{R}_5$)).

\vspace{.1in}
Let us end the section with some notation and needed statements. 
In the entire paper, $\h$ denotes a separable Hilbert space and $(e_k)_{k=1}^\infty$ denotes an orthonormal basis for $\h$. Recall that 
a sequence $(f_k)_{k=1}^\infty$ with elements from $\h$ is called a {\it Bessel sequence in $\h$} if there exists a positive constant $B$ so that 
$\sum_{k=1}^\infty |\langle h, f_k\rangle|^2 \leq B\|h\|^2$ for every $h\in\h$. 
Given a Bessel sequence $F=(f_k)_{k=1}^\infty$ in $\h$, the synthesis operator $T_F$ given by $T_F (c_k)=\sum_{k=1}^\infty c_k f_k$ is well defined and bounded from $\ell^2$ into $\h$ (see, e.g., \cite[Theorem 3.2.3]{Cbook}).

\section{Removing a condition from ($\mathcal{R}_6$)}

In this section we obtain a characterization of Riesz bases with relaxed conditions compare to the ones in ($\mathcal{R}_6$). 
We show that one (any one) of the asumptions for completeness in ($\mathcal{R}_6$) can be omitted. 
First let us prove that one can remove the assumption for completeness of $(f_k)_{k=1}^\infty$:

\begin{thm}\label{th1}
 For a sequence $(f_k)_{k=1}^\infty$ in $\h$, the following conditions are equivalent:
\begin{itemize}

\item[{\rm (i)}] $(f_k)_{k=1}^\infty$ is a Riesz basis for H.

\item[{\rm (ii)}]  $(f_k)_{k=1}^\infty$ is a Bessel sequence in $\h$ and it has a biorthogonal sequence $(g_k)_{k=1}^\infty$ which is 
a complete Bessel sequence in $\h$.
\end{itemize}
\end{thm}
\begin{proof} 
 (ii) $\Rightarrow$ (i): 
Let $T_F$ and $T_G$ denote the synthesis operators of the Bessel sequences
$(f_k)_{k=1}^\infty$ and $(g_k)_{k=1}^\infty$, respectively. 

First observe that $T_F$ is injective. Indeed, if $T_F(c_k)_{k=1}^\infty=0$ for some $(c_k)_{k=1}^\infty\in\ell^2$, then for every $j\in{\mathbb N}$ we have
$0=\langle \sum_{k=1}^\infty c_k f_k, g_j\rangle = c_j$.

For the surjectivity of $T_F$, take an arbitrary $f\in \h$ and observe that for every $j\in\mathbb N$ one has
\begin{equation}\label{fg1}
\langle T_F T_G^*f, g_j\rangle = 
\sum_{k=1}^\infty \langle f, g_k\rangle \langle f_k, g_j\rangle 
=\langle f, g_j\rangle, 
\end{equation}
which by the completeness of $(g_j)_{j=1}^\infty$ imples that $f= T_F T_G^*f$. 
Thus, $T_F$ is surjective.

For every $k\in \mathbb N$,  $f_k=T_F\delta_k$. Then $(f_k)_{k=1}^\infty$ 
is the image of an orthonormal basis under a bounded bijective operator, which is precisely 
the definition of a Riesz basis. 

 (i) $\Rightarrow$ (ii): follows from Theorem \ref{rbequiv}.
\end{proof}

Let us now consider relations between biorthogonality and completeness. 
It is welll known that completeness of one of two biorthogonal sequences does not imply  completeness of the other one:

\begin{ex} \label{exYoung} {\rm \cite{Young81}}
Let $\h$ be infinite-dimensional. Take an infinite-dimensional closed proper subspace $K$ of $\h$, an orthonormal basis $(g_k)_{k=1}^\infty$ for $K$, and an orthonormal basis $\mathcal{B}$ for $K^\perp$. Let $(y_k)_{k=1}^\infty$ be a sequence which contains every element of 
 $\mathcal{B}$ infinitely many times. Consider $f_k:=g_k+y_k$, $k\in\mn$. 
Then   $(f_k)_{k=1}^\infty$ is complete in $\h$, $(f_k)_{k=1}^\infty$ and $(g_k)_{k=1}^\infty$ are biorthogonal, and clearly  $(g_k)_{k=1}^\infty$ is not complete in $\h$.  As a simple concrete case, consider  $(g_k)$ being $(e_n)_{n=2}^\infty$ and $(f_k)$ being $(e_n+e_1)_{n=2}^\infty$.
\end{ex}
 
Here we observe that adding the Bessel assumption leads to a different conclusion regarding completeness. We show that Bessel biorthogonal sequences are either both complete or both incomplete:

\begin{prop} \label{prop1} Let $(f_k)_{k=1}^\infty$ and $(g_k)_{k=1}^\infty$ be Bessel sequences in $\h$ which are biorthogonal. If one of them is complete in $\h$, then the other one is also complete in $\h$.
\end{prop}
\begin{proof} Without loss of generality, assume that $(g_k)_{k=1}^\infty$ is complete in $\h$.
Fix an arbitrary $f\in \h$. As in the proof of Theorem \ref{th1}, for every $f\in\h$ and every $j\in\mathbb N$, the biorthogonalify assumption leads to validity of (\ref{fg1}) 
and then the completeness of $(g_k)_{k=1}^\infty$ in $\h$ implies  that 
$f= T_F T_G^*f= \sum_{k=1}^\infty \langle f, g_k\rangle f_k$. Therefore, $(f_k)_{k=1}^\infty$ is complete in $\h$.
\end{proof}

Note that in the above statement it is essential to assume the Bessel condition for both sequences. If only one of the biorthogonal sequences is assumed to be Bessel, the conclusion of Proposition \ref{prop1} does not hold in general. 
Consider for example the Bessel sequence $(g_k)_{k=1}^\infty$ and the non-Bessel sequence $(f_k)_{k=1}^\infty$ in Example \ref{exYoung}. 
Notice also that the Bessel property is not necessary for completeness of the biorthogonal sequences. Consider for example the complete biorthogonal sequences $(\frac{1}{k}e_k)_{k=1}^\infty$ and  $(k e_k)_{k=1}^\infty$, only one of which is Bessel, or 
$(e_1, 2e_2, \frac{1}{3} e_3, 4e_4, \frac{1}{5}e_5, \ldots)$ and $(e_1,  \frac{1}{2} e_2,$ $ 3e_3, \frac{1}{4}e_4, 5e_5,\ldots)$,  none of which is Bessel.

\vspace{.1in}

As a consequence of Proposition \ref{prop1} and Theorem \ref{th1}, one can now write  the following characterization of Riesz bases,  
removing one (any one) of the completeness conditions from ($\mathcal{R}_6$).

\begin{thm} \label{thmnew}
 For a sequence $(f_k)_{k=1}^\infty$ in $\h$, the following conditions are equivalent:
\begin{itemize}

\item[{\rm (i)}] $(f_k)_{k=1}^\infty$ is a Riesz basis for H.

\item[{\rm (ii)}]  $(f_k)_{k=1}^\infty$ is a Bessel sequence in $\h$, it has a biorthogonal sequence $(g_k)_{k=1}^\infty$ which is also a Bessel sequence in $\h$, and one of $(f_k)_{k=1}^\infty$ and $(g_k)_{k=1}^\infty$  is complete in $\h$.
\end{itemize}
\end{thm}

\noindent
{\bf Applications.} 
For many structured systems, completeness is not a simple property to check.  
From this point of view, 
the removal of a verification for completeness from  ($\mathcal{R}_6$) can be very useful for simpler verification of the Riesz basis property.

On the other hand, Theorem \ref{thmnew} can be used to get conclusions  which can not be derived from the equivalence of ($\mathcal{R}_1$)  and ($\mathcal{R}_6$). 
  Consider for example the Gaussian $g(x) = e^{-\pi x^2}$, a
sequence $\Lambda$ of 
uniformly separated points from ${\mathbb R}^2$,  and the associated Gabor system
$G_{g,\Lambda} = (e^{2\pi i \mu\cdot} g(\cdot-\tau))_{(\tau,\mu)\in \Lambda}$. 
 It is well known that $G_{g,\Lambda}$  
 is a Bessel sequence (follows e.g. from \cite[Sec.\,5.1]{Heil2007}) but not a Riesz basis for $L^2(\mathbb R)$, see, e.g. \cite{SeipI92}. 
 However, for certain $\Lambda$, $G_{g,\Lambda}$ is a complete minimal system in $L^2(\mathbb R)$ - this is for example the case when $\Lambda=\{(-1,0), (1,0)\}\cup \{ (0, \pm\sqrt{2n}):n=1,2,3,\ldots\} \cup \{ (\pm\sqrt{2n},0):n=1,2,3,\ldots\}$ \cite{ALS2009}, or 
  when  $\Lambda$ is the lattice $\mathbb Z \times \mathbb Z$ without the point $(1,0)$,   see, e.g. \cite{LS1999}. 
For such $\Lambda$'s, based on Theorem \ref{rbequiv}, one can conclude that the biorthogonal system of $G_{g,\Lambda}$ is non-Bessel or non-complete,  
while Theorem \ref{thmnew} immediately implies that the biorthogonal system is non-Bessel. 
We also expect further applications of  Theorem \ref{thmnew} for  
other structured systems.

\vspace{.1in}\noindent
{\bf Acknowledgments} 
The author  acknowledges support from
the Austrian Science Fund (FWF) through the 
Project ``FLAME" Y 551-N13
and from the Vienna Science and Technology Fund (WWTF) through Project VRG12-009.

\vspace{.1in}
\noindent
Diana T. Stoeva\\
Austrian Academy of Sciences\\ 
Acoustics Research Institute\\
Wohllebengasse 12-14, Vienna 1040, Austria\\
dstoeva@kfs.oeaw.ac.at

\end{document}